\documentclass[letterpaper, 10 pt, conference]{cssconf}
\IEEEoverridecommandlockouts  
\usepackage[utf8]{inputenc}
\usepackage[T1]{fontenc}
\usepackage[english]{babel}
\usepackage{cite}
\usepackage{graphicx}
\usepackage{times}
\usepackage{xcolor}
\usepackage{amsfonts,amsmath,amssymb} 
\usepackage{url}
\usepackage{siunitx}
\usepackage{mathtools}
\usepackage{subfig}
\graphicspath{{\Figures}}

\graphicspath{{Figures/}}
\newtheorem{theorem}{Theorem}

\newtheorem{remark}{Remark}
\newtheorem{lemma}{Lemma}

\newtheorem{definition}{Definition}

\DeclareMathOperator{\sign}{sign}
\newcommand{\abs}[1]{\left\lvert#1\right\rvert}
\newcommand{\bigabs}[1]{\big\lvert#1\big\rvert}

\newcommand{\norm}[1]{\left\lVert#1\right\rVert}

\newcommand{\Rset}{\mathbb{R}}
\newcommand{\eps}{\varepsilon}
\newcommand{\Oinf}[1]{\mathcal{O}_{\infty}(\varepsilon^{#1})}

\DeclareMathOperator{\modfun}{mod}
\DeclareMathOperator{\wrap}{w}
\DeclareMathOperator{\bigO}{\mathcal{O}}
\newcommand{\umax}{u_{m}}

\newcommand{\tdot}{\, \widetilde{\cdot} \,}

\title{\LARGE \bf Adding virtual measurements by PWM-induced signal injection}
\author{Dilshad Surroop\textsuperscript{1,2}, Pascal Combes\textsuperscript{2}, Philippe Martin\textsuperscript{1} and Pierre Rouchon\textsuperscript{1}
	\thanks{\textsuperscript{1}~D.~Surroop, P.~Martin and P.~Rouchon are with the Centre Automatique et Systèmes, MINES ParisTech, PSL Research University, Paris, France
		{\tt\footnotesize\{dilshad.surroop, philippe.martin,pierre.rouchon\}@mines-paristech.fr}}%
	\thanks{\textsuperscript{2}~D.~Surroop and P.~Combes  are with Schneider Toshiba Inverter Europe, Pacy-sur-Eure, France
		{\tt\footnotesize  pascal.combes@se.com}}
}

\begin{document}
\maketitle

\begin{abstract}
We show that for PWM-operated devices, it is possible to benefit from signal injection \emph{without an external probing signal}, by suitably using the excitation provided by the PWM itself. As in the usual signal injection framework conceptualized in~\cite{CombeJMMR2016ACC}, an extra ``virtual measurement'' can be made available for use in a control law, but without the practical drawbacks caused by an external signal. 
\end{abstract}

\section{Introduction}\label{sec:introduction}
Signal injection is a control technique which consists in adding a fast-varying probing signal to the control input. This excitation creates a small ripple in the measurements, which contains useful information if properly decoded. The idea was introduced in~\cite{JanseL1995ITIA,CorleL1998ITIA} for controlling electric motors at low velocity using only measurements of currents. It was later conceptualized in~\cite{CombeJMMR2016ACC} as a way of producing ``virtual measurements'' that can be used to control the system, in particular to overcome observability degeneracies. Signal injection is a very effective method, see e.g. applications to electromechanical devices along these lines in~\cite{JebaiMMR2016IJC,YiOZ2018SCL}, but it comes at a price: the ripple it creates may in practice yield unpleasant acoustic noise and excite unmodeled dynamics, in particular in the very common situation where the device is fed by a Pulse Width Modulation (PWM) inverter; indeed, the frequency of the probing signal may not be as high as desired so as not to interfere with the PWM (typically, it can not exceed~\SI{500}{\hertz} in an industrial drive with a \SI{4}{\kilo\hertz}-PWM frequency).

The goal of this paper is to demonstrate that for PWM-operated devices, it is possible to benefit from signal injection \emph{without an external probing signal}, by suitably using the excitation provided by the PWM itself, as e.g. in\cite{WangX2004TPE}.
More precisely, consider the Single-Input Single-Output system 
\begin{subequations}\label{eq:SISO}
	\begin{IEEEeqnarray}{rCl}
		\dot x &=& f(x) + g(x)u,\label{eq:dynamics}\\
		y &=& h(x),
	\end{IEEEeqnarray}
\end{subequations}
where $u$ is the control input and $y$ the measured output. We first show in section~\ref{sec:pwm} that when the control is impressed through PWM, the dynamics may be written as
\begin{IEEEeqnarray}{rCl}
	\dot x &=& f(x) + g(x)\bigl(u+s_0(u,\tfrac{t}{\varepsilon}) \bigr),\label{eq:s0dynamics}
\end{IEEEeqnarray}
with $s_0$ $1$-periodic and zero-mean in the second argument, i.e. $s_0(u,\sigma+1)=s_0(u,\sigma)$ and $\int_0^1s_0(u,\sigma)\,d\sigma=0$ for all~$u$; $\varepsilon$ is the PWM period, hence assumed small. The difference with usual signal injection is that the probing signal $s_0$ generated by the modulation process now depends not only on time, but also on the control input~$u$. This makes the situation more complicated, in particular because $s_0$ can be discontinuous in both its arguments. Nevertheless, we show in section~\ref{sec:averaging} that the second-order averaging analysis of~\cite{CombeJMMR2016ACC} can be extended to this case. In the same way, we show in section~\ref{sec:demodulation} that the demodulation procedure of~\cite{CombeJMMR2016ACC} can be adapted to make available the so-called virtual measurement
\begin{IEEEeqnarray*}{rCl}
	y_v &:=& H_1(x) := \varepsilon h'(x)g(x),
\end{IEEEeqnarray*}
in addition to the actual measurement $y_a := H_0(x):= h(x)$. This extra signal is likely to simplify the design of a control law, as illustrated on a numerical example in section~\ref{sec:example}.

Finally, we list some definitions used throughout the paper; $S$ denotes a function of two variables, which is $T$-periodic in the second argument, i.e. $S(v,\sigma+T)=S(v,\sigma)$ for all~$v$:
\begin{itemize}
	\item the mean of~$S$ in the second argument is the function (of one variable) $\overline{S}(v):=\tfrac{1}{T}\int_0^{T}S(v,\sigma)d\sigma$; $S$ has zero mean in the second argument if $\overline{S}$ is identically zero
	\item if $S$ has zero mean in the second argument, its zero-mean primitive in the second argument is defined by
	\begin{IEEEeqnarray*}{rCl}
		S_1(v,\tau) &:=& \int_{0}^{\tau}S(v,\sigma)d\sigma -  \frac{1}{T}\int_{0}^{T}\int_{0}^{\tau}S(v,\sigma)d\sigma d\tau;
	\end{IEEEeqnarray*}
	notice $S_1$ is $T$-periodic in the second argument because $S$ has zero mean in the second argument
	\item the moving average $M(k)$ of~$k$ is defined by
	\begin{IEEEeqnarray*}{rCl}
		M(k)(t) &:=& \frac{1}{\varepsilon}\int_{t-\eps}^{t}k(\tau)d\tau
	\end{IEEEeqnarray*}
	\item 
	$\mathcal{O}_\infty$ denotes the uniform ``big O'' symbol of analysis, namely $f(z,\varepsilon)=\Oinf{p}$ if $\abs{f(z,\varepsilon)}\le K\varepsilon^p$ for $\eps$ small enough, with $K>0$ independent of $z$ and $\varepsilon$.
\end{itemize}

\section{PWM-induced signal injection}\label{sec:pwm}
When the control input~$u$ in~\eqref{eq:dynamics} is impressed through a PWM process with period~$\varepsilon$, the resulting dynamics reads
\begin{IEEEeqnarray}{rCl}
	\dot x &=& f(x) + g(x) \mathcal{M}\bigl(u,\tfrac{t}{\varepsilon}\bigr),\label{eq:Mdynamics}
\end{IEEEeqnarray}
with $\mathcal{M}$ $1$-periodic and mean~$u$ in the second argument;
the detailed expression for $\mathcal{M}$ is given below. Setting  $s_0(u,\sigma):=\mathcal{M}(u,\sigma)-u$, \eqref{eq:Mdynamics} obviously takes the form~\eqref{eq:s0dynamics}, with $s_0$ 1-periodic and zero-mean in the second argument.

\begin{figure}
	\includegraphics[width=\columnwidth]{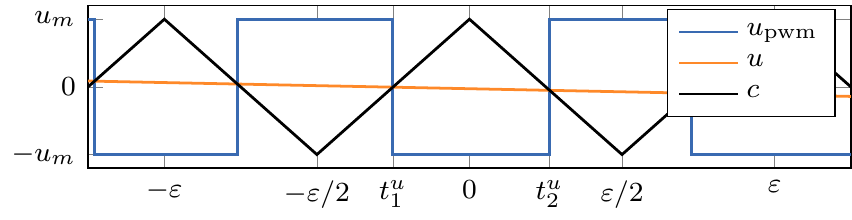}
	\caption{PWM: $u$ is compared to $c$ to produce $u_{\rm pwm}$.}
	\label{fig:pwm}
\end{figure}


Classical PWM with period~$\varepsilon$ and range~$[-\umax,\umax]$ is obtained by comparing the input signal~$u$ to the $\varepsilon$-periodic sawtooh carrier defined by
\begin{IEEEeqnarray*}{rCl}
	c(t) &:=& 
	\begin{dcases*} 
		\umax+4\wrap\bigl(\tfrac{t}{\varepsilon}\bigr) & if $-\frac{\umax}{2}\le\wrap\bigl(\frac{t}{\varepsilon}\bigr)\le0$\\ 
		\umax-4\wrap\bigl(\tfrac{t}{\varepsilon}\bigr) & if $0\le\wrap(\frac{t}{\varepsilon})\le\frac{\umax}{2}$;
	\end{dcases*}
\end{IEEEeqnarray*}
the 1-periodic function $\wrap(\sigma):=\umax\modfun(\sigma+\frac{1}{2},1)-\frac{\umax}{2}$ wraps the normalized time~$\sigma=\frac{t}{\varepsilon}$ to $[-\frac{\umax}{2},\frac{\umax}{2}]$. If $u$ varies slowly enough, it crosses the carrier~$c$ exactly once on each rising and falling ramp, at times $t_1^u<t_2^u$ such that
\begin{IEEEeqnarray*}{rCl}
	u(t_1^u) &=& \umax+4\wrap\Bigl(\frac{t_1^u}{\varepsilon}\Bigr)\\
	u(t_2^u) &=& \umax-4\wrap\Bigl(\frac{t_2^u}{\varepsilon}\Bigr).
\end{IEEEeqnarray*}
\begin{figure}
	\centering
		\hspace*{0.6em}\includegraphics{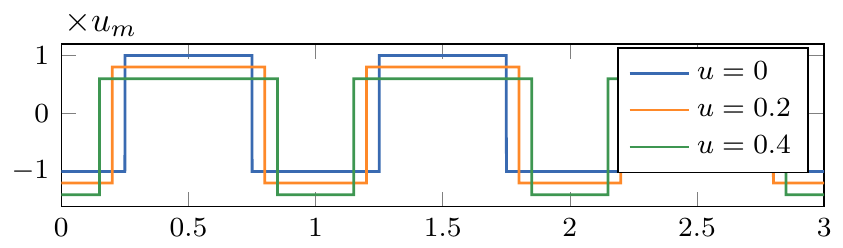}
		\includegraphics{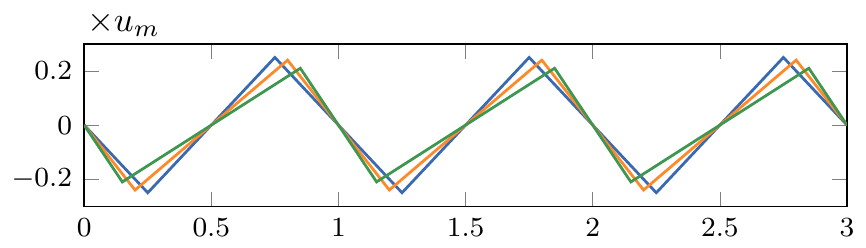}
		\includegraphics{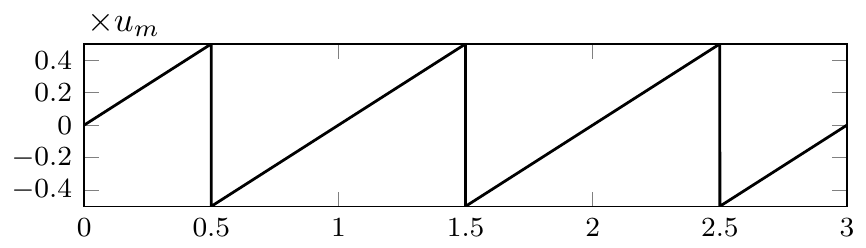}
	\caption{$s_0(u,\cdot)$ (top) and $s_1(u,\cdot)$ (middle) for $u=0, 0.2, 0.4$; $\wrap$ (bottom).}
\end{figure}

The PWM-encoded signal is therefore given by
\begin{IEEEeqnarray*}{rCl}
	u_{\rm pwm}(t)&=&
	\begin{dcases*} 
		\umax & if $-\frac{\umax}{2}<\wrap\bigl(\frac{t}{\varepsilon}\bigr)\le\wrap\bigl(\frac{t_1^u}{\varepsilon}\bigr)$\\ 
		-\umax & if $\wrap\bigl(\frac{t_1^u}{\varepsilon}\bigr)<\wrap\bigl(\frac{t}{\varepsilon}\bigr)\le\wrap\bigl(\frac{t_2^u}{\varepsilon}\bigr)$\\ 
		\umax & if $\wrap\bigl(\frac{t_2^u}{\varepsilon}\bigr)<\wrap\bigl(\tfrac{t}{\varepsilon}\bigr)\le\frac{\umax}{2}$.
	\end{dcases*}
\end{IEEEeqnarray*}
Fig.~\ref{fig:pwm} illustrates the signals $u$, $c$ and~$u_{\rm pwm}$.
The function 
\begin{IEEEeqnarray*}{rCl}
	\mathcal{M}(u,\sigma)&:=& 
	\begin{dcases*} 
		\umax & if $-2\umax<4\wrap(\sigma)\le u-\umax$\\ 
		-\umax & if $u - \umax<4\wrap(\sigma)\le \umax -u$\\ 
		\umax & if $\umax-u<4\wrap(\sigma)\le2\umax$
	\end{dcases*}\\
	&=& \umax+\sign\bigl(u-\umax-4\wrap(\sigma)\bigr)\\
	&&\quad+\>\sign\bigl(u-\umax+4\wrap(\sigma)\bigr),
\end{IEEEeqnarray*}
which is obviously 1-periodic and with mean~$u$ with respect to its second argument, therefore completely describes the PWM process since 
$u_{\rm pwm}(t)=\mathcal{M}\bigl(u(t),\frac{t}{\varepsilon}\bigr)$.

Finally, the induced zero-mean probing signal is
\begin{IEEEeqnarray*}{rCl}
	s_0(u,\sigma) &:=& \mathcal{M}(u,\sigma)-u\\
	&=& \umax-u+\sign\bigl(\tfrac{u-\umax}{4} - \wrap(\sigma)\bigr)\\
	&&\quad+\>\sign\bigl(\tfrac{u-\umax}{4}+\wrap(\sigma)\bigr),
\end{IEEEeqnarray*}
and its zero-mean primitive in the second argument is
\begin{IEEEeqnarray*}{rCl}
	s_1(u,\sigma) &:=& \bigl(1-\tfrac{u}{\umax}\bigr)\wrap(\sigma) - \bigabs{\tfrac{u-\umax}{4}-\wrap(\sigma)} \\
	&&+\, \abs{\tfrac{u-\umax}{4}+\wrap(\sigma)}.
\end{IEEEeqnarray*}
\begin{remark}\label{rem:regularity}
	As $s_0$ is only piecewise continuous, one might expect problems to define the ``solutions'' of~\eqref{eq:s0dynamics}. But as noted above, if the input $u(t)$ of the PWM encoder varies slowly enough, its output $u_{\rm pwm}(t)=\mathcal{M}\bigl(u(t),\frac{t}{\varepsilon}\bigr)$ will have exactly two discontinuities per PWM period. Chattering is therefore excluded, which is enough to ensure the existence and uniqueness of the solutions of~\eqref{eq:s0dynamics}, see~\cite{LehmaB1996TPE}, without the need for the more general Filipov theory~\cite{Filip1988book}. Of course, we assume (without loss of generality in practice) that $f$, $g$ and $h$ in~\eqref{eq:SISO} are smooth enough.

	Notice also $s_1$ is continuous and piecewise~$C^1$ in both its arguments. The regularity in the second argument was to be expected as $s_1(u,\cdot)$ is a primitive of~$s_0(u,\cdot)$; on the other hand, the regularity in the first argument stems from the specific form of~$s_0$.
\end{remark}

\section{Averaging and PWM-induced injection}\label{sec:averaging}
Section~\ref{sec:mainResult} outlines the overall approach and states the main Theorem~\ref{theorem:main}, which is proved in the somewhat technical section~\ref{mainProof}. As a matter of fact, the proof can be skipped without losing the main thread; suffice to say that if $s_0$ were Lipschitz in the first argument, the proof would essentially be an extension of the analysis by "standard" second-order averaging of~\cite{CombeJMMR2016ACC}, with more involved calculations

\subsection{Main result}\label{sec:mainResult}
Assume we have designed a suitable control law 
\begin{IEEEeqnarray*}{rCl}
	\overline{u} &=& \alpha(\overline{\eta},\overline{Y},t) \\
	\dot {\overline{\eta}} &=& a(\overline{\eta},\overline{Y},t),
\end{IEEEeqnarray*}
where $\overline{\eta} \in \Rset^q$, for the system
\begin{IEEEeqnarray*}{rCl}
	\dot{\overline{x}} &=& f(\overline{x}) + g(\overline x)\overline{u},\\
	\overline{Y} &=& H(\overline{x}) 
	:= \begin{pmatrix} h(\overline{x})\\ \varepsilon h'(\overline{x})g(\overline{x}) \end{pmatrix}.
\end{IEEEeqnarray*}
By ``suitable'', we mean the resulting closed-loop system
\begin{subequations}
	\label{eq:claveraged}
	\begin{IEEEeqnarray}{rCl}
		\dot{\overline{x}} &=& f(\overline{x}) + g(\overline x)\alpha\bigl( \overline{\eta},H(\overline{x}),t \bigr) \\
		\dot{\overline{\eta}} &=& a \bigl( \overline{\eta}, H(\overline{x}),t \bigr)
	\end{IEEEeqnarray}
\end{subequations}
has the desired exponentially stable behavior. We have changed the notations of the variables with $\overline{\,\cdot\,}$ to easily distinguish between the solutions of \eqref{eq:claveraged} and of \eqref{eq:modified_system} below. Of course, this describes an unrealistic situation:
\begin{itemize}
	\item PWM is not taken into account
	\item the control law is not implementable, as it uses not only the actual output $\overline y_a=h(\overline{x})$, but also the a priori not available virtual output $\overline y_v=\varepsilon h'(\overline{x})g(\overline{x})$.
\end{itemize}

Define now (up to $\scriptstyle{\Oinf{2}}$) the function
\begin{multline}\label{eq:Hbar}
	\overline{H}(x,\eta,\sigma,t):=\\
	H\biggl(x-\varepsilon g(x) s_1\Bigl(\alpha\bigl(\eta,H(x),t\bigr),\sigma\Bigr)\biggr) + \Oinf{2},
\end{multline}
where $s_1$ is the zero-mean primitive of~$s_0$ in the second argument, and consider the control law
\begin{subequations}
	\label{eq:modifiedlaw}
	\begin{IEEEeqnarray}{rCl}
		u &=& \alpha\bigl(\eta,\overline{H}(x,\eta,\tfrac{t}{\varepsilon},t),t\bigr)\\
		\dot \eta &=& a\bigl(\eta,\overline{H}(x,\eta,\tfrac{t}{\varepsilon},t),t\bigr).
	\end{IEEEeqnarray}
\end{subequations}
The resulting closed-loop system, including PWM, reads
\begin{subequations}
	\label{eq:modified_system}
	\begin{IEEEeqnarray}{rCl}
		\dot x &=& f(x) + 
		g(x)\mathcal{M}\Bigl(\alpha\bigl(\eta,\overline{H}(x,\eta,\tfrac{t}{\varepsilon},t),t\bigr),\tfrac{t}{\varepsilon}\Bigr)
		\IEEEeqnarraynumspace\\
		\dot \eta &=& a\bigl(\eta,\overline{H}(x,\eta,\tfrac{t}{\varepsilon},t),t\bigr).
	\end{IEEEeqnarray}
\end{subequations}
Though PWM is now taken into account, the control law~\eqref{eq:modifiedlaw} still seems to contain unknown terms. Nevertheless, it will turn out from the following result that it can be implemented.
\begin{theorem}\label{theorem:main}
	Let $(x(t),\eta (t))$ be the solution of~\eqref{eq:modified_system} starting from $(x_0,\eta_0)$, and define
	$u(t):=\alpha\bigl( \eta(t), H\bigl(x(t)\bigr), t \bigr)$ and $y(t):= H\bigl(x(t)\bigr)$;
	let $(\overline x(t), \overline \eta(t))$ be the solution of~\eqref{eq:claveraged} starting from $\bigl(x_0-\varepsilon g(x_0)s_1\bigl(u(0),0\bigl),\eta_0\bigr)$, and define $\overline u(t):=\alpha\bigl( \overline \eta(t), H\bigl(\overline x(t)\bigr), t \bigr)$. Then, for all $t\ge 0$,
	\begin{subequations}
		\begin{IEEEeqnarray}{rCl}
			\label{subeq:x}
			x(t) &=& \overline x(t) + \varepsilon g\bigl(\overline x(t)\bigr)s_1\bigl(\overline u(t),\tfrac{t}{\varepsilon}\bigr) + \Oinf{2} \\
			\label{subeq:eta}
			\eta(t) &=& \overline \eta (t) + \Oinf{2} \\
			\label{subeq:y}
			y(t) &=& H_0\bigl(\overline x(t)\bigr) + H_1\bigl(\overline x(t)\bigr) s_1\bigl(\overline u(t),\tfrac{t}{\varepsilon}\bigr) + \Oinf{2}.\IEEEeqnarraynumspace
		\end{IEEEeqnarray} 
	\end{subequations}
\end{theorem}\vspace{2mm}

The practical meaning of the theorem is the following. As the solution $\bigl(x(t),\eta(t))\bigr)$ is piecewise~$C^1$, we have by Taylor expansion using~\eqref{subeq:x}-\eqref{subeq:eta} that
$u(t)=\overline u(t)+\Oinf{2}$. In the same way, as $s_1$ is also piecewise~$C^1$, we have
\begin{IEEEeqnarray*}{rCl}
	s_1\bigl(u(t),\tfrac{t}{\varepsilon}\bigr)
	&=& s_1\bigl(\overline u(t),\tfrac{t}{\varepsilon}\bigr) + \Oinf{2}.
\end{IEEEeqnarray*} 
As a consequence, we can invert~\eqref{subeq:x}-\eqref{subeq:eta}, which yields
\begin{subequations}
	\label{eq:inverse_xbar_etabar}
\begin{IEEEeqnarray}{rCl}
	\overline x(t) &=& x(t) - \varepsilon g\bigl(x(t)\bigr)s_1\bigl(u(t),\tfrac{t}{\varepsilon}\bigr) + \Oinf{2} \\
	\overline\eta(t) &=& \eta(t)+\Oinf{2}.
\end{IEEEeqnarray}
\end{subequations}
Using this into~\eqref{eq:Hbar}, we then get
\begin{IEEEeqnarray}{rCl}
	\IEEEeqnarraymulticol{3}{l}{
	\overline{H}\bigl(x(t),\eta(t),\tfrac{t}{\varepsilon},t\bigr)}\nonumber\\*\qquad
		&=& H\Bigl(x(t)-\varepsilon g\bigl(x(t)\bigr) s_1\bigl(u(t),\tfrac{t}{\varepsilon}\bigr)\Bigr) + \Oinf{2},\IEEEeqnarraynumspace\nonumber\\
		&=& H\bigl(\overline x(t)\bigr)+ \Oinf{2}.\label{eq:HbarHxbar}
\end{IEEEeqnarray}
On the other hand, we will see in section~\ref{sec:demodulation} that, thanks to~\eqref{subeq:y}, we can produce  an estimate $\widehat{Y}=H(\overline{x})+\Oinf{2}$. This means the PWM-fed dynamics~\eqref{eq:Mdynamics} acted upon by the implementable feedback
\begin{IEEEeqnarray*}{rCl}
	u &=& \alpha(\eta,\widehat{Y},t)\\
	\dot \eta &=& a(\eta,\widehat{Y},t).
\end{IEEEeqnarray*}
behaves exactly as the ``ideal'' closed-loop system~\eqref{eq:claveraged}, except for the presence of a small ripple (described by~\eqref{subeq:x}-\eqref{subeq:eta}).

\begin{remark}\label{rem:regularity2}
	Notice that, according to Remark~\ref{rem:regularity}, $H_0\bigl(\overline x(t)\bigr)$ and $H_1\bigl(\overline x(t)\bigr)$ in~\eqref{subeq:y} may be as smooth as desired (the regularity is inherited from only $f,g,h,\alpha,a$); on the other hand, $s_1\bigl(u(t),\tfrac{t}{\varepsilon}\bigr)$ is only continuous and piecewise~$C^1$. Nevertheless, this is enough to justify all the Taylor expansions performed in the paper.
\end{remark}

\subsection{Proof of Theorem \ref{theorem:main}}\label{mainProof}
Because of the lack of regularity of~$s_0$, we must go back to the fundamentals of the second-order averaging theory presented in~\cite[chapter 2]{SandeVM2005book} (with slow time dependence \cite[section 3.3]{SandeVM2005book}). We first introduce two ad hoc definitions.
\begin{definition}
	A function $\varphi(X,\sigma)$ is \emph{slowly-varying in average} if there exists $\lambda >0$ such that for $\varepsilon$ small enough,
		\begin{IEEEeqnarray*}{rCl}
			\int_a^{a+T}\hspace{-1.1em} \norm{\varphi \bigl(p(\varepsilon\sigma) + \varepsilon^k q(\sigma),\sigma\bigr) - \varphi \bigl(p(\varepsilon\sigma), \sigma\bigr)}  d\sigma \leq \lambda T\varepsilon^k,
		\end{IEEEeqnarray*}
		where $p$, $q$ are continuous with $q$ bounded; $a$ and $T>0$ are arbitrary constants. Notice that if $\varphi$ is Lipschitz in the first variable then it is slowly-varying in average. The interest of this definition is that it is satisfied by $s_0$. 
\end{definition}
\begin{definition}
	A function $\phi$ is \emph{$\Oinf{3}$ in average} if there exists $K>0$ such that $\norm{\int_0^\sigma \phi\bigl(q(s),s\bigr) \, ds} \leq K \varepsilon^3 \sigma$ for all $\sigma \geq 0$. Clearly, if $\phi$ is $\Oinf{3}$ then it is $\Oinf{3}$ in average.
\end{definition}

The proof of Theorem \ref{theorem:main} follows the same steps as \cite[chapter 2]{SandeVM2005book}, but with weaker assumptions. 
We first rewrite~\eqref{eq:modified_system} in the fast timescale $\sigma:=t/\eps$ as
\begin{IEEEeqnarray}{rCl}
	\label{eq:system_proof}
	\frac{dX}{d\sigma} &=& \varepsilon F(X,\sigma,\eps\sigma).
\end{IEEEeqnarray}
where $X := (x, \eta)$ and
	\begin{IEEEeqnarray*}{rCl}
		F(X,\sigma,\tau) &:=& \begin{pmatrix}
			f(x)+g(x)\mathcal{M}\Bigl(\alpha\bigl(\eta,\overline H(x,\eta,\sigma,\tau),\tau\bigr)\Bigr) \\
			a\bigl(\eta,\overline H (x,\eta,\sigma,\tau),\tau\bigr)
		\end{pmatrix}. \\
	\end{IEEEeqnarray*}
Notice $F$ is $1$-periodic in the second argument. Consider also the so-called averaged system 
\begin{IEEEeqnarray}{rCl}\label{eq:averagedSystem}
	\frac{d\overline{X}}{d\sigma} &=& \varepsilon\overline{F}(\overline{X},\eps\sigma).
\end{IEEEeqnarray}
where $\overline{F}$ is the mean of $F$ in the second argument.

Define the near-identity transformation 
\begin{IEEEeqnarray}{rCl}
	\label{eq:transformation}
X &=& \widetilde X + \varepsilon W(\widetilde X,\sigma,\eps\sigma),
\end{IEEEeqnarray}
where $\widetilde X := (\widetilde x, \widetilde \eta)$ and
	\begin{IEEEeqnarray*}{rCl}
				W(\widetilde X, \sigma,\tau) :=  
		\begin{pmatrix}
			g(\widetilde x) \\
			0
		\end{pmatrix}s_1\Bigl(\alpha\bigl(\widetilde\eta,H(\widetilde x,\widetilde \eta,\sigma,\tau),\tau\bigr),\sigma\Bigr).
	\end{IEEEeqnarray*}
	Inverting~\eqref{eq:transformation} yields
\begin{IEEEeqnarray}{rCl}
	\label{eq:inversion_x_tilde}
	\widetilde X = X - \varepsilon W(X,\sigma,\varepsilon\sigma) + \Oinf{2}.
\end{IEEEeqnarray}
By lemma~\ref{lemma:transformation}, this transformation puts \eqref{eq:system_proof} into
\begin{IEEEeqnarray}{rCl}
	\label{eq:system_lemma_1}
	\frac{d\widetilde{X}}{d\sigma} &=& \eps\overline F(\widetilde{X},\eps\sigma) + \eps^2\Phi(\widetilde X, \sigma, \eps\sigma)+\phi(\widetilde X,\sigma,\eps\sigma);\IEEEeqnarraynumspace
\end{IEEEeqnarray}
$\Phi$ is periodic and zero-mean in the second argument, and slowly-varying in average, and $\phi$ is $\Oinf{3}$ in average.

By lemma~\ref{lemma:estimation}, the solutions  $\overline X(\sigma)$ and $\widetilde X(\sigma)$ of \eqref{eq:averagedSystem} and~\eqref{eq:system_lemma_1}, starting from the same initial conditions, satisfy
\begin{IEEEeqnarray*}{rCl}
	\widetilde X(\sigma) &=& \overline X(\sigma) + \Oinf{2}.
\end{IEEEeqnarray*}

As a consequence, the solution $X(\sigma)$ of \eqref{eq:system_proof} starting from $X_0$ and the solution $\overline X(\sigma)$ of \eqref{eq:averagedSystem} starting from $X_0 - \varepsilon W(X_0,0,0)$ are related by $X(\sigma) = \overline X(\sigma) + \varepsilon W\bigl(\overline X(\sigma),\sigma,\varepsilon\sigma\bigr) + \Oinf{2}$, which is exactly \eqref{subeq:x}-\eqref{subeq:eta}. Inserting \eqref{subeq:x} in $y = h(x)$ and Taylor expanding yields \eqref{subeq:y}.

\begin{remark}
	If $s_0$ were differentiable in the first variable, $\Phi$ would be Lipschitz and $\phi$ would be $\Oinf{3}$ in \eqref{eq:system_lemma_1}, hence the averaging theory of \cite{SandeVM2005book} would directly apply.
\end{remark}

\begin{remark}
	In the sequel, we prove for simplicity only the estimation $\widetilde X(\sigma) = \overline X(\sigma) + \mathcal{O}(\varepsilon^2)$ on a timescale $1/\varepsilon$. The continuation to infinity follows from the exponential stability of \eqref{eq:claveraged}, exactly as in \cite[Appendix]{CombeJMMR2016ACC}. 

	In the same way, lemma~\ref{lemma:estimation} is proved without slow-time dependence, the generalization being obvious as in \cite[section 3.3]{SandeVM2005book}. 
\end{remark}

\begin{lemma}
	\label{lemma:transformation}
	The transformation~\eqref{eq:transformation} puts \eqref{eq:system_proof} into \eqref{eq:system_lemma_1}, where $\Phi$ is periodic and zero-mean in the second argument, and slowly-varying in average, and $\phi$ is $\Oinf{3}$ in average. 

\end{lemma}

\begin{proof}
	To determine the expression for $d\widetilde X/d\sigma$, the objective is to compute $dX/d\sigma$ as a function of $\widetilde X$ with two different methods. On the one hand we replace $X$ with its transformation~\eqref{eq:transformation} in the closed-loop system~\eqref{eq:system_proof}, and on the other hand we differentiate~\eqref{eq:transformation} with respect to $\sigma$.

	We first compute $s_0(\alpha(\eta,\overline H(x,\eta,\sigma,\varepsilon\sigma),\varepsilon\sigma),\sigma)$ as a function of $\widetilde X = (\widetilde x,\widetilde \eta)$. Exactly as in \eqref{eq:HbarHxbar}, with $(\widetilde x, \widetilde \eta)$ replacing $(\overline x, \overline \eta)$, and \eqref{eq:inversion_x_tilde} replacing \eqref{eq:inverse_xbar_etabar}, we have
\begin{IEEEeqnarray*}{rCl}
	\overline H(x,\eta,\sigma,\varepsilon\sigma) &=& H(\widetilde x) + \Oinf{2}.
\end{IEEEeqnarray*}
Therefore, by Taylor expansion
\begin{IEEEeqnarray*}{rCl}
	\alpha(\eta,\overline H(x,\eta,\sigma,\varepsilon\sigma), \varepsilon\sigma) &=& \alpha(\widetilde \eta, H(\widetilde x),\varepsilon\sigma) \\
										     &&+\, \varepsilon^2 K_\alpha (\widetilde X, \sigma),
\end{IEEEeqnarray*}
with $K_\alpha$ bounded. The lack of regularity of $s_0$ prevents further Taylor expansion; nonetheless, we still can write 
\begin{multline*}
	s_0\bigl(\alpha(\eta, \overline H(x,\eta,\sigma,\varepsilon\sigma),\varepsilon \sigma),\sigma\bigr) = \\
	\hspace{2cm} s_0\bigl(\alpha(\widetilde\eta, H(\widetilde x), \varepsilon\sigma) + \varepsilon^2 K_\alpha(\widetilde X,\sigma),\sigma\bigr). 
\end{multline*}
Finally, inserting \eqref{eq:transformation} into \eqref{eq:system_proof} and Taylor expanding, yields after tedious but straightforward computations, 
\begin{IEEEeqnarray}{rCl}
	\label{eq:dXdsigma2}
	\frac{dX}{d\sigma} &=& \varepsilon \overline F (\widetilde X,\varepsilon\sigma) + \varepsilon G(\widetilde X) s_0^{\alpha,+} (\tdot) \notag \\
			   &&+\, \varepsilon^2 \overline F (\widetilde X,\varepsilon\sigma)G(\widetilde X) s_1^\alpha(\tdot) \notag \\
			   &&+\, \varepsilon^2 G'(\widetilde X) G(\widetilde X) s_1^\alpha (\tdot) s_0^{\alpha,+} (\tdot) + \Oinf{3}; \IEEEeqnarraynumspace 
\end{IEEEeqnarray}
we have introduced the following notations
\begin{IEEEeqnarray*}{rCl}
	\label{eq:delta_s0}
	(\tdot) &:=& (\widetilde X,\sigma,\varepsilon\sigma) \notag \\
	s_i^{\alpha} \bigl(\tdot \bigr) &:=& s_i\bigl(\alpha(\widetilde\eta,H(\widetilde x),\varepsilon\sigma), \sigma\bigr), \notag \\
	s_0^{\alpha,+} \bigl(\tdot\bigr) &:=& s_0\bigl(\alpha(\widetilde\eta,H(\widetilde x),\varepsilon\sigma) + \varepsilon^2 K_\alpha(\widetilde X, \sigma),\sigma\bigr) \notag \\
	\Delta s_0^\alpha (\tdot) &:=& s_0^{\alpha,+}(\tdot) - s_0 (\tdot) \\
	G(X) &:=& \begin{pmatrix}
		g(x) \\  
		0
	\end{pmatrix} \\
	\overline F(X,\varepsilon\sigma) &:=& \begin{pmatrix}
f(x)+g(x)\alpha\bigl(\eta,H(x),\varepsilon\sigma\bigr) \\
			a\bigl(\eta, H(x),\varepsilon\sigma\bigr)
	\end{pmatrix}.
\end{IEEEeqnarray*}

We now time-differentiate \eqref{eq:transformation}, which reads with the previous notations
\begin{IEEEeqnarray*}{rCl}
	X &=& \widetilde X + \varepsilon G(\widetilde X) s_1^\alpha (\tdot).
\end{IEEEeqnarray*}
This yields
\begin{IEEEeqnarray}{rCl}
	\label{eq:dXdsigmaDiff}
	\frac{dX}{d\sigma} &=& \frac{d\widetilde X}{d\sigma} + \varepsilon G'(\widetilde X) \frac{d\widetilde X}{d\sigma} s_1^\alpha (\tdot) + \varepsilon G(\widetilde X)\partial_1 s_1^\alpha (\tdot) \frac{d\widetilde X}{d\sigma} \notag \\
			   &&+\, \varepsilon G(\widetilde X) s_0^\alpha (\tdot) + \varepsilon^2 G(\widetilde X) \partial_3 s_1^\alpha (\tdot), \IEEEeqnarraynumspace
\end{IEEEeqnarray}
since $\partial_2 s_1^\alpha = s_0^\alpha$. Now assume $\widetilde X$ satisfies 
\begin{IEEEeqnarray}{rCl}
	\label{eq:dXdsigmaClaim}
	\frac{d\widetilde X}{d\sigma} &=& \varepsilon \overline F (\widetilde X,\varepsilon\sigma) + \varepsilon G(\widetilde X) \Delta s_0^\alpha (\tdot) + \varepsilon^2 \Psi (\tdot),
\end{IEEEeqnarray} 
where $\Psi(\tdot)$ is yet to be computed. Inserting \eqref{eq:dXdsigmaClaim} into \eqref{eq:dXdsigmaDiff},  
\begin{IEEEeqnarray}{rCl}
	\label{eq:dXdsigma1}
	\frac{dX}{d\sigma} &=& \varepsilon \overline F (\widetilde X,\varepsilon\sigma) + \varepsilon G(\widetilde X) \Delta s_0^\alpha (\tdot) + \varepsilon^2 \Psi(\tdot) \notag \\
			   &&+\, \varepsilon^2 G'(\widetilde X) \overline F (\widetilde X,\varepsilon\sigma) s_1^\alpha(\tdot)\notag  \\
			   &&+\, \varepsilon^2 G'(\widetilde X) G(\widetilde X) \Delta s_0^\alpha (\tdot) s_1^\alpha (\tdot) \notag \\
			   &&+\, \varepsilon^2 G(\widetilde X) \partial_1 s_1^\alpha (\tdot) \overline F (\widetilde X,\varepsilon\sigma) \notag \\
			   &&+\, \varepsilon^2 G(\widetilde X) \partial_1 s_1^\alpha (\tdot) G(\widetilde X) \Delta s_0^\alpha(\tdot) \notag \\
			   &&+\, \varepsilon G(\widetilde X) s_0^\alpha (\tdot) \notag \\
			   &&+\, \varepsilon^2 G(\widetilde X) \partial_3 s_1^\alpha (\tdot) \notag \\
			   &&+\, \Oinf{3}.
\end{IEEEeqnarray}

Next,  equating \eqref{eq:dXdsigma1} and \eqref{eq:dXdsigma2}, $\Psi$ satisfies
\begin{IEEEeqnarray}{rCl}
	\label{eq:psi}
	\Psi (\tdot) &=& [\overline F,G](\widetilde X,\varepsilon\sigma)s_1^\alpha (\tdot) + G'(\widetilde X) G(\widetilde X) s_0^\alpha(\tdot) s_1^\alpha (\tdot) \notag \\
		     &&-\, G(\widetilde X) \partial_1 s_1^\alpha (\tdot) \overline F (\widetilde X,\varepsilon\sigma) -\, G(\widetilde X) \partial_3 s_1^\alpha (\tdot)  \notag\\
		     &&-\, G(\widetilde X) \partial_1 s_1^\alpha (\tdot) G(\widetilde X) \Delta s_0^\alpha(\tdot).
\end{IEEEeqnarray}
This gives the expressions of $\Phi$ and $\phi$ in \eqref{eq:system_lemma_1},
\begin{IEEEeqnarray*}{rCl}
\Phi (\tdot) &:=& [\overline F,G](\widetilde X,\varepsilon\sigma)s_1^\alpha (\tdot) + G'(\widetilde X) G(\widetilde X) s_0^\alpha(\tdot) s_1^\alpha (\tdot) \notag \\
		     &&-\, G(\widetilde X) \partial_1 s_1^\alpha (\tdot) \overline F (\widetilde X,\varepsilon\sigma) -\, G(\widetilde X) \partial_3 s_1^\alpha (\tdot),  \notag\\
\phi(\tdot) &:=& \varepsilon^2 \Psi_1(\tdot) + \varepsilon G(\widetilde X) \Delta s_0^\alpha,
\end{IEEEeqnarray*}
with
\begin{IEEEeqnarray*}{rCl}
	\Psi_1(\tdot) := -G(\widetilde X)\partial_1s_1^\alpha (\tdot) G(\widetilde X)\Delta s_0^\alpha (\tdot).
\end{IEEEeqnarray*}

The last step is to check that $\Phi$ and $\phi$ satisfy the assumptions of the lemma. Since $s_0^\alpha$, $s_1^\alpha$, $\partial_1 s_1^\alpha$ and $\partial_3 s_1^\alpha$ are periodic and zero-mean in the second argument, and slowly-varying in average,  so is $\Phi$. There remains to prove that $\phi = \Oinf{3}$ in average. Since $\Delta s_0^\alpha$ is slowly-varying in average, 
\begin{IEEEeqnarray*}{rCl}
	\label{eq:hyp_delta_s0}
	\int_0^\sigma \norm{\Delta s_0^\alpha (\tdot(s))}ds  \leq \lambda_0\sigma\varepsilon^2.
\end{IEEEeqnarray*}
with $\lambda_0>0$. $G$ being bounded by a constant $c_g$, this implies 
\begin{IEEEeqnarray*}{rCl}
	\norm{\int_0^\sigma \varepsilon G(\widetilde X(s)) \Delta s_0^\alpha(\tdot(s)) \, ds } \leq c_g \lambda_0 \sigma \varepsilon^3. 
\end{IEEEeqnarray*}
Similarly, $\partial_1 s_1$ being bounded by $c_{11}$, $\Psi_1$ satisfies 
\begin{IEEEeqnarray*}{rCl}
	\norm{\int_0^\sigma \varepsilon^2\Psi_1 (\tdot(s)) \, ds} \leq c_g^2 c_{11} \lambda_0 \sigma \varepsilon_0 \varepsilon^3.
\end{IEEEeqnarray*}
Summing the two previous inequalities yields 
\begin{IEEEeqnarray*}{rCl}
	\norm{\int_0^\sigma \phi(\tdot(s)) \,ds} \leq \lambda_0 c_g(1+c_{11}c_g\varepsilon_0)\sigma\varepsilon^3, 
\end{IEEEeqnarray*}
which concludes the proof. 
\end{proof}

\begin{lemma}
	\label{lemma:estimation}
	Let $\overline X(\sigma)$ and $\widetilde X(\sigma)$ be respectively the solutions of \eqref{eq:averagedSystem} and~\eqref{eq:system_lemma_1} starting at $0$ from the same initial conditions. Then, for all $\sigma\ge0$
			\begin{IEEEeqnarray*}{rCl}
				\widetilde X(\sigma) &=& \overline X(\sigma) + \Oinf{2}.
			\end{IEEEeqnarray*}
\end{lemma}

\begin{proof}
Let $E(\sigma) := \widetilde{X}(\sigma) - \overline{X}(\sigma)$. Then,
\begin{IEEEeqnarray*}{rCl}
	E(\sigma) &=& \int_{0}^\sigma \Bigl[\frac{d\widetilde{X}}{d\sigma}(s) - \frac{d\overline{X}}{d\sigma}(s)\Bigr] \, ds \\
		  &=& \varepsilon \int_{0}^\sigma \bigl[F(\widetilde{X}(s)) - F(\overline{X}(s)) \bigr] \, ds  \\
		  &&+\, \varepsilon^2 \int_{0}^\sigma \Phi(\tdot(s))  \, ds + \int_{0}^\sigma \phi(\tdot(s)) \, ds
\end{IEEEeqnarray*}
As $F$ is Lipschitz with constant $\lambda_F$,
\begin{IEEEeqnarray*}{rCl}
	\varepsilon\int_{0}^\sigma \norm{F(\widetilde X(s)) - F(\overline X(s)) }  ds \leq \varepsilon \lambda_F \int_{0}^\sigma \norm{E(s)}\, ds. 
\end{IEEEeqnarray*}
On the other hand, there exists by lemma~\ref{lemma:besjes} $c_1$ such that 
	\begin{IEEEeqnarray*}{rCl}
		\varepsilon^2\norm{\int_{0}^\sigma \Phi(\tdot(s)) \, ds} \leq c_1\varepsilon^2
	\end{IEEEeqnarray*}
	Finally, as $\phi$ is $\Oinf{3}$ in average, there exists $c_2$ such that
	\begin{IEEEeqnarray*}{rCl}
		\norm{\int_{0}^\sigma \phi(\tdot(s)) \, ds} \leq c_2\varepsilon^3 \sigma.
	\end{IEEEeqnarray*}
	The summation of these estimations yields 
		\begin{IEEEeqnarray*}{rCl}
			\norm{E(\sigma)} \leq \varepsilon \lambda_F \int_{0}^\sigma \norm{E(s)}\, ds + c_1 \varepsilon^2 + c_2\varepsilon^3 \sigma. 
		\end{IEEEeqnarray*}
	Then by Gronwall's lemma \cite[Lemma 1.3.3]{SandeVM2005book}
\begin{IEEEeqnarray*}{rCl}
	\norm{E(\sigma)} \leq \left(\frac{c_2}{\lambda_F} + c_1 \right) e^{\lambda_F \sigma} \varepsilon^2,
\end{IEEEeqnarray*}
which means $\widetilde X = \overline X + \Oinf{2}$. 
\end{proof}

The following lemma is an extension of Besjes' lemma \cite[Lemma 2.8.2]{SandeVM2005book} when $\varphi$ is no longer Lipschitz, but only slowly-varying in average. 

\begin{lemma}
	\label{lemma:besjes}
	Assume $\varphi (X,\sigma)$ is $T$-periodic and zero-mean in the second argument, bounded, and slowly-varying in average.  
	Assume the solution $X(\sigma)$ of $\dot X = \Oinf{}$ is defined for $0 \leq \sigma \leq L/\varepsilon$. There exists $c_1 > 0$ such that 
	\begin{IEEEeqnarray*}{rCl}
		\norm{\int_0^\sigma \varphi (X(s),s) \, ds} \leq c_1. 
	\end{IEEEeqnarray*}
\end{lemma}

\begin{proof}
	Along the lines of \cite{SandeVM2005book}, we divide the interval $[0,t]$ in $m$ subintervals $[0,T],\ldots,[(m-1)T,mT]$ and a remainder $[mT,t]$. By splitting the integral on those intervals, we write
	\begin{IEEEeqnarray*}{rCl}
		\int_0^\sigma \varphi(x(s),s) \, ds &=& \sum_{i=0}^m \int_{(i-1)T}^{iT} \varphi(x((i-1)T),s)\, ds  \\
					       &&\hspace{-2cm}+\sum_{i=0}^m \int_{(i-1)T}^{iT} \bigr[\varphi(x(s),s) - \varphi(x((i-1)T),s) \bigl] \, ds \\
					       && \hspace{-2cm}+ \int_{mT}^\sigma \varphi(x(s),s) \, ds,
	\end{IEEEeqnarray*}
	where each of the integral in the first sum are zero as $\varphi$ is periodic with zero mean. Since $\varphi$ is bounded, the remainder is also bounded by a constant $c_2 > 0$. Besides
	\begin{IEEEeqnarray*}{rCl}
		x(s) &=& x((i-1)T) + \int_{(i-1)T}^s \dot x (\tau) \, d\tau  \\
		     &=& x((i-1)T) + \varepsilon q(s),
	\end{IEEEeqnarray*}
	with $q$ continuous and bounded. By hypothesis, there exists $\lambda > 0$ such that for $0 \leq i \leq m$,
\begin{IEEEeqnarray*}{rCl}
	\int_{(i-1)T}^{iT} \norm{\varphi(x(s),s) - \varphi(x((i-1)T),s)} \, ds  \leq  \lambda T \varepsilon 
\end{IEEEeqnarray*}
Therefore by summing the previous estimations,
\begin{IEEEeqnarray*}{rCl}
	\norm{\int_0^\sigma \varphi(x(s),s) \, ds}  \leq m \lambda T\varepsilon + c_2,
\end{IEEEeqnarray*}
with $mT \leq t \leq L/\varepsilon$, consequently $m\lambda T \varepsilon +c_2 \leq \lambda L + c_2$; which concludes the proof. 
\end{proof}

\section{Demodulation}\label{sec:demodulation}
From~\eqref{subeq:y}, we can write the measured signal~$y$ as
\begin{IEEEeqnarray*}{rCl}	
	y(t) &=& y_a(t)+ y_v(t)s_1\bigl(u(t),\tfrac{t}{\eps}\bigr)+\Oinf{2},
\end{IEEEeqnarray*}
where the signal $u$ feeding the PWM encoder is known. The following result shows $y_a$ and $y_v$ can be estimated from~$y$, for use in a control law as described in section~\ref{sec:mainResult}. 

\begin{theorem}
	Consider the estimators $\widehat y_a$ and $\widehat y_v$ defined by
	\begin{IEEEeqnarray*}{rCl}
		\widehat y_a(t) &:=& \frac{3}{2} M(y)(t) - \frac{1}{2}M(y)(t-\varepsilon)\\\\
		k_\Delta(\tau) &:=& \bigl(y(\tau)-\widehat y_a(\tau)\bigr)s_1\bigl(u(\tau),\tfrac{\tau}{\varepsilon}\bigr)\\
		\widehat y_v(t) &:=& \frac{M(k_\Delta)(t)}{\overline{s_1^2}\bigl(u(t)\bigr)},
	\end{IEEEeqnarray*}
	where $M:y\mapsto \varepsilon^{-1}\int_0^\varepsilon y(\tau)d\tau$ is the moving average operator, and $\overline{s_1^2}$ the mean of $s_1^2$ in the second argument (cf end of section~\ref{sec:introduction}). Then,
\begin{subequations}	
	\begin{IEEEeqnarray}{rCl}
		\widehat y_a(t) &=& y_a(t) + \Oinf{2} \label{eq:estya}\\
		\widehat y_v(t) &=& y_v(t) + \Oinf{2}. \label{eq:estyv}
	\end{IEEEeqnarray}
\end{subequations}
\end{theorem}
Recall that by construction $y_v(t)=\Oinf{}$, hence \eqref{eq:estyv} is essentially a first-order estimation; notice also that
$\overline{s_1^2}\bigl(u(t)\bigr)$ is always non-zero when $u(t)$ does not exceed the range of the PWM encoder.

\begin{proof}
Taylor expanding $y_a$, $y_v$, $u$ and~$s_1$ yields
\begin{IEEEeqnarray*}{rCl}
	y_a(t -\tau) &=& y_a(t) -\tau\dot y_a(t) + \bigO_\infty(\tau^2)\\
	y_v(t -\tau) &=& y_v(t) + \Oinf{}\bigO_\infty(\tau)\\	
	s_1\bigl(u(t-\tau),\sigma\bigr) &=& s_1\bigl(u(t)+\bigO_\infty(\tau),\sigma\bigr)\\
	&=& s_1\bigl(u(t),\sigma\bigr)+\bigO_\infty(\tau);
\end{IEEEeqnarray*}
in the second equation, we have used $y_v(t)=\Oinf{}$. 
The moving average of $y_a$ then reads 
\begin{IEEEeqnarray}{rCl}	
	\label{eq:Mya}
	M(y_a)(t) &=& \frac{1}{\varepsilon}\int_{0}^{\eps}y_a(t-\tau)d\tau\notag \\
	&=& \frac{1}{\varepsilon}\int_{0}^{\eps}\bigl(y_a(t)-\tau\dot y_a(t)+\bigO_\infty(\tau^2)\bigr)d\tau\notag \\
	&=& y_a(t) - \frac{\eps}{2}\dot y_a(t) + \Oinf{2}. 
\end{IEEEeqnarray}
A similar computation for $k_v(t):=y_v(t) s_1\bigl(u(t),\tfrac{t}{\varepsilon}\bigr)$ yields
\begin{IEEEeqnarray}{rCl}
	\label{eq:Myv}
	M(k_v)(t) &=& \frac{1}{\varepsilon}\int_{0}^{\eps}y_v(t-\tau)s_1\bigl(u(t-\tau),\tfrac{t-\tau}{\eps}\bigr)\,d\tau\notag\\
	%
	&=& y_v(t)\Bigl(\overline{s_1}\bigl(u(t)\bigr) + \Oinf{} \Bigr) +\Oinf{2}\notag\\
	&=& \Oinf{2},
\end{IEEEeqnarray}
since $s_1$ is 1-periodic and zero mean in the second argument. Summing \eqref{eq:Mya} and \eqref{eq:Myv}, we eventually find
\begin{IEEEeqnarray*}{rCl}	
	M(y)(t) &=& y_a(t) - \frac{\eps}{2}\dot y_a(t) + \Oinf{2}.
\end{IEEEeqnarray*}
As a consequence, we get after another Taylor expansion
\begin{IEEEeqnarray*}{rCl}	
	\frac{3}{2}M(y)(t)-\frac{1}{2}M(y)(t-\eps) &=& y_a(t) + \Oinf{2},\end{IEEEeqnarray*}
which is the desired estimation~\eqref{eq:estya}.

On the other hand, \eqref{eq:estya} implies
\begin{IEEEeqnarray*}{rCl}	
	k_\Delta(t)
	&=& y_v(t)s_1^2\bigl(u(t),\tfrac{t}{\eps}\bigr)+\Oinf{2}.
\end{IEEEeqnarray*}
Proceeding as for $M(k_v)$, we find
\begin{IEEEeqnarray*}{rCl}
	M(k_\Delta)(t) &=& \frac{1}{\varepsilon}\int_{0}^{\eps}y_v(t-\tau)s_1^2\bigl(u(t-\tau),\tfrac{t-\tau}{\eps}\bigr)\,d\tau\notag\\
	&=& y_v(t)\Bigl(\overline{s_1^2}\bigl(u(t)\bigr) + \Oinf{} \Bigr) +\Oinf{2}\notag\\
	&=& y_v(t)\overline{s_1^2}\bigl(u(t)\bigr) + \Oinf{2}.
\end{IEEEeqnarray*}
Dividing by $\overline{s_1^2}\bigl(u(t)\bigr)$ yields the desired estimation~\eqref{eq:estyv}.
\end{proof}

\section{Numerical example}\label{sec:example}

\begin{figure}[t]
	\captionsetup[subfloat]{farskip=2pt,captionskip=2pt}
	\raggedleft
	\subfloat[State $x_1$, reference $x_1^{\rm ref}$, and virtual measurement $y_v$.]{\includegraphics{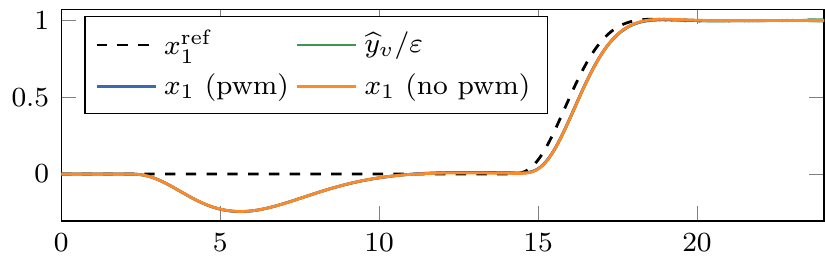}}\\
	\subfloat[State $x_2$.]{\includegraphics{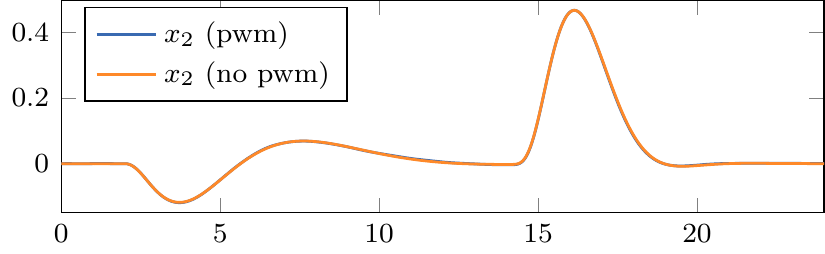}}\\
	\subfloat[State $x_3$.]{\includegraphics{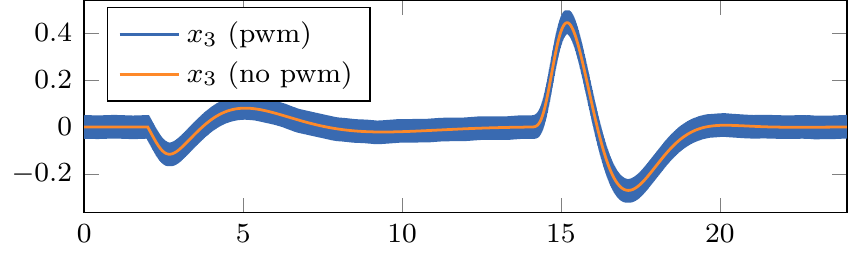}}\\
	\subfloat[State $x_3$, zoom on (c).]{\includegraphics{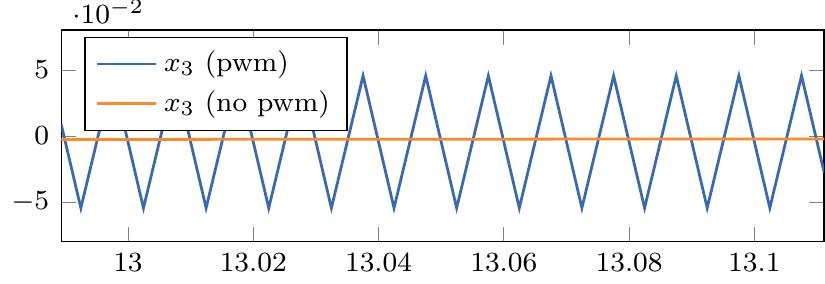}}
	\caption{States $x_1,x_2,x_3$ with ideal and actual control laws.}
	\label{fig:states}
\end{figure}

\begin{figure}[t]
	\includegraphics{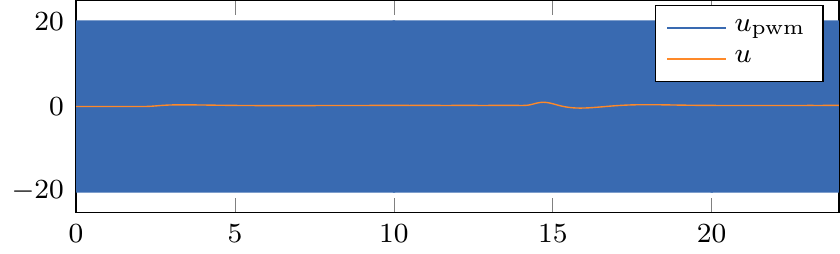}
	\includegraphics{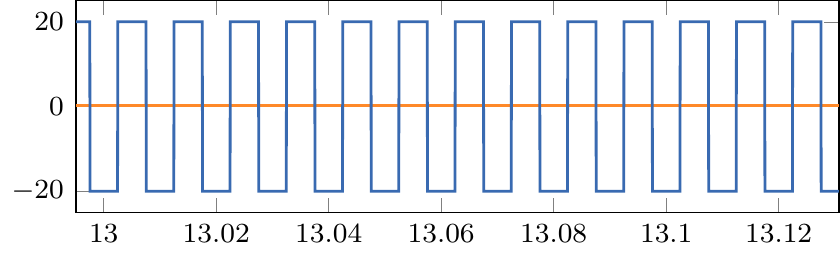}
	\caption{Control input $u$ and its modulation $u_{\rm pwm}$; full view (top), zoom (bottom).}
	\label{fig:u-upwm}
\end{figure}
\begin{figure}[ht]
	\raggedleft
	\includegraphics{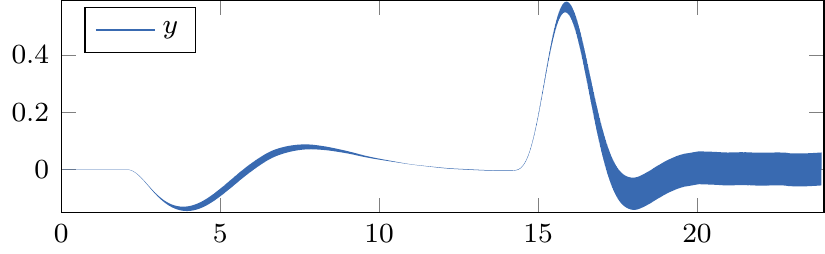}
	\includegraphics{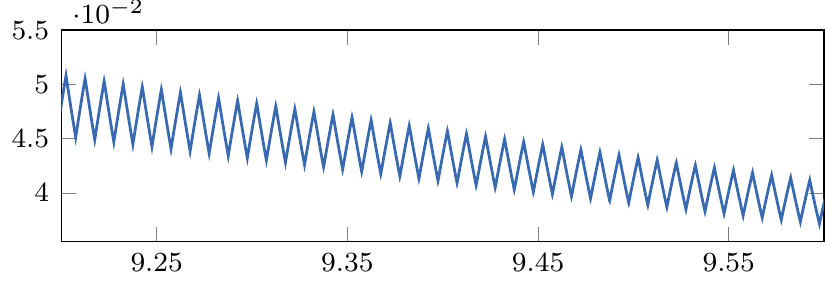}
	\caption{Measured output $y$ (top); full view (top), zoom (bottom).}
	\label{fig:y}
\end{figure}

We illustrate the interest of the approach on the system
\begin{IEEEeqnarray*}{rCl}	
	\dot x_1 &=& x_2, \\
	\dot x_2 &=& x_3, \\
	\dot x_3 &=& u + d, \\
	y &=& x_2 + x_1 x_3,
\end{IEEEeqnarray*}
where $d$ is an unknown disturbance; $u$ will be impressed through PWM with frequency~\SI{1}{\kilo\hertz} (i.e. $\varepsilon=\num{e-3}$) and range~$[-20,20]$. The objective is to control $x_1$, while rejecting the disturbance $d$, with a response time of a few seconds. We want to operate around equilibrium points, which are of the form $(x_1^{\rm eq}, 0, 0; -d^{\rm eq},d^{\rm eq} )$, for $x_1^{\rm eq}$ and $d^{\rm eq}$ constant. Notice the observability degenerates at such points, which renders not trivial the design of a control law. 

Nevertheless the PWM-induced signal injection makes available the virtual measurement
\begin{IEEEeqnarray*}{C}
	y_v = \varepsilon\begin{pmatrix}x_3& 1& x_1\end{pmatrix}\begin{pmatrix}0\\0\\1\end{pmatrix}  = \varepsilon x_1,
\end{IEEEeqnarray*}
from which it is easy to design a suitable control law, without even using the actual input $y_a=x_2+x_1x_3$. The system being now fully linear, we use a classical controller-observer, with disturbance estimation to ensure an implicit integral effect. The observer is thus given by
\begin{IEEEeqnarray*}{rCl}
	\dot{\widehat x}_1 &=& \widehat x_2 + l_1\bigl(\tfrac{y_v}{\varepsilon} - \widehat x_1\bigr),\\
	\dot{\widehat x}_2 &=& \widehat x_3+ l_2\bigl(\tfrac{y_v}{\varepsilon} - \widehat x_1\bigr),\\
	\dot{\widehat x}_3 &=& u + \widehat d + l_3\bigl(\tfrac{y_v}{\varepsilon} - \widehat x_1\bigr),\\
	\dot{\widehat d}       &=& l_d\bigl(\tfrac{y_v}{\varepsilon} - \widehat x_1\bigr),
\end{IEEEeqnarray*}
and the controller by
\begin{IEEEeqnarray*}{rCl}
	u = -k_1 \widehat x_1-k_2 \widehat x_2-k_3 \widehat x_3 - k_d \widehat d + kx_1^{\rm ref}.
\end{IEEEeqnarray*}
The gains are chosen to place the observer eigenvalues at $(-1.19,-0.73,-0.49 \pm 0.57i)$ and the controller eigenvalues at $(-6.59,-3.30 \pm 5.71i)$. The observer is slower than the controller in accordance with dual Loop Transfer Recovery, thus ensuring a reasonable robustness. Setting $\eta:=(\widehat x_1,\widehat x_2,\widehat x_3,\widehat d)^T$, this controller-observer obviously reads
\begin{subequations}
	\label{eq:idealLaw}
	\begin{IEEEeqnarray}{rCl}
		u &=& -K \eta+ kx_1^{\rm ref} \\
		\dot \eta &=& M\eta + Nx_1^{\rm ref}(t) + Ly_v 
	\end{IEEEeqnarray}
\end{subequations}
Finally, this ideal control law is implemented as
\begin{subequations}
	\label{eq:actualLaw}
	\begin{IEEEeqnarray}{rCl}
		u_{\rm pwm}(t) &=& \mathcal{M}\bigl(-K\eta+ kx_1^{\rm ref},\tfrac{t}{\varepsilon}\bigr) \\
		\dot \eta &=& M\eta + Nx_1^{\rm ref} + L \tfrac{\widehat y_v}{\varepsilon},
	\end{IEEEeqnarray}
\end{subequations}
where $\mathcal{M}$ is the PWM function described in section~\ref{sec:pwm}, and $\widehat y_v$ is obtained by the demodulation process of section~\ref{sec:demodulation}.

The test scenario is the following: at $t=0$, the system start at rest at the origin; from $t=2$, a disturbance $d=-0.25$ is applied to the system; at $t=14$, a filtered unit step is applied to the reference $x_1^{\rm ref}$. In Fig.~\ref{fig:states} the ideal control law~\eqref{eq:idealLaw}, i.e. without PWM and assuming $y_v$ known, is compared to the true control law~\eqref{eq:actualLaw}: the behavior of~\eqref{eq:actualLaw} is excellent, it is nearly impossible to distinguish the two situations on the responses of $x_1$ and $x_2$ as by~\eqref{subeq:x} the corresponding ripple is only $\Oinf{2}$; the ripple is visible on~$x_3$, where it is~$\Oinf{}$. The corresponding control signals $u$ and $u_{\rm pwm}$ are displayed in Fig.~\ref{fig:u-upwm}, and the corresponding measured outputs in Fig.~\ref{fig:y}.

To investigate the sensitivity to measurement noise, the same test was carried out with band-limited white noise (power density~\num{1e-9}, sample time~\num{1e-5}) added to~$y$. Even though the ripple in the measured output is buried in noise, see Fig.~\ref{fig:y_noise}, the virtual output is correctly demodulated and the control law~\eqref{eq:actualLaw} still behaves very well.  

\begin{figure}[ht]
	\raggedleft
	\includegraphics{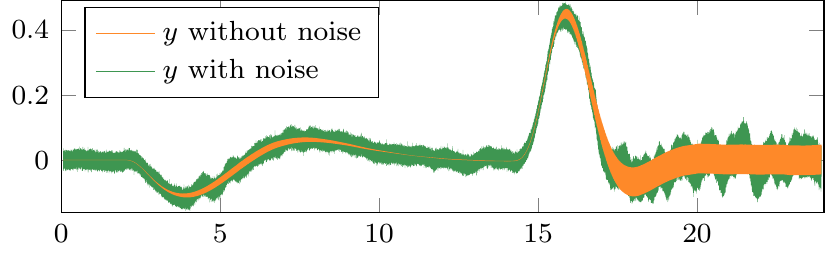}
	\includegraphics{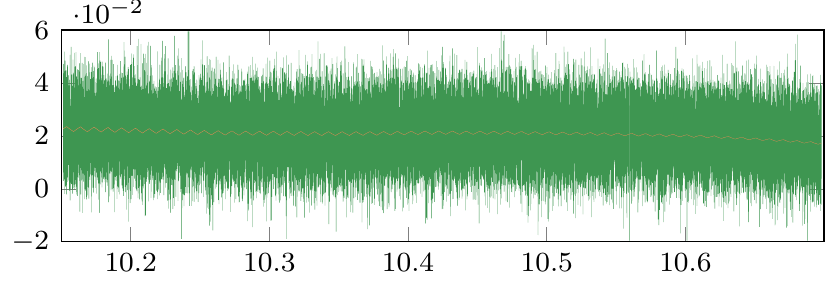}
	\caption{Measured output $y$ with and without noise; full view (top), zoom (bottom).}
	\label{fig:y_noise}
\end{figure}

\begin{figure}[ht]
	\raggedleft
	\includegraphics{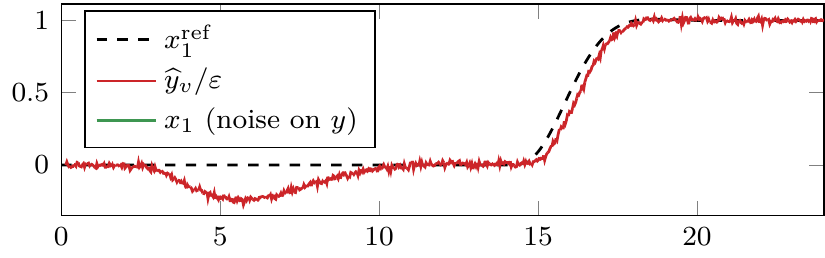}
	\caption{$x_1^{\rm ref}$, $x_1$, and $\widehat y_v$ in the presence of noise.}
	\label{fig:measurement_noise}
	\label{fig:x1_noise}
\end{figure}

\section*{Conclusion}
We have presented a method to take advantage of the benefits of signal injection in PWM-fed systems without the need for an external probing signal. For simplicity, we have restricted to Single-Input Single-Output systems, but there are no essential difficulties to consider Multiple-Input Multiple-Output systems. Besides, though we have focused on classical PWM, the approach can readily be extended to arbitrary modulation processes, for instance multilevel PWM; in fact, the only requirements is that $s_0$ and $s_1$ meet the regularity assumptions discussed in remark~\ref{rem:regularity}.


\bibliographystyle{phmIEEEtran}
\bibliography{biblio.bib}
\end{document}